\documentclass[oneside, 12pt]{amsart}

\usepackage[margin=1in]{geometry}
\usepackage{amsthm}
\usepackage{enumitem}  
\usepackage[OT2,T1]{fontenc}
\usepackage{amscd, amssymb, enumitem, amsmath, mathrsfs, amsfonts, amsaddr}
\usepackage{url}
\usepackage{tikz}
\usepackage{fullpage}
\usepackage{microtype}
\usetikzlibrary{decorations.pathreplacing}
\usepackage[backref=page]{hyperref}
\usepackage{hyperref}
\usepackage[utf8]{inputenc}

\DeclareSymbolFont{cyrletters}{OT2}{wncyr}{m}{n}
\DeclareMathSymbol{\Sha}{\mathalpha}{cyrletters}{"58}

\newtheorem{theorem}{Theorem}[section]

\newtheorem{lemma}[theorem]{Lemma}

\newtheorem*{proposition*}{Proposition}
\newtheorem{corollary}[theorem]{Corollary}
\newtheorem*{question}{Question 1}

\theoremstyle{definition}
\newtheorem{definition}[theorem]{Definition}
\newtheorem{example}[theorem]{Example}
\newtheorem{conjecture}[theorem]{Conjecture}

\newtheorem{remark}[theorem]{Remark}
\newtheorem{emptyremark}[theorem]{}
\newtheorem*{acknowledgement}{Acknowledgement}

\theoremstyle{remark}

\title{Torsion and twists of abelian varieties}
\author{Mentzelos Melistas}
\date{\today}

\address{Charles University, Faculty of Mathematics and Physics, Department of
Algebra, Sokolov\-sk\' a 83, 18600 Praha~8, Czech Republic}

\address{University of Twente, Department of Applied Mathematics, Drienerlolaan 5, 7522 NB Enschede, The Netherlands}

\begin{document}

\maketitle

\begin{abstract}
    In this article, we investigate the possible torsion subgroups of twists of abelian varieties with good reduction. As an application, we prove a theorem concerning ramified primes over any quadratic extension where odd-order torsion growth is achieved. In particular, we show that for every rational elliptic curve and every imaginary quadratic field not equal to $\mathbb{Q}(\sqrt{-3})$ satisfying the Heegner hypothesis no odd-order torsion growth can occur.
\end{abstract}

\section{Introduction}

Let $K$ be a number field of degree $d$ and let $A/K$ be an abelian variety of dimension $g$. The Mordell-Weil theorem states that the set $A(K)$ of $K$-rational points of $A/K$ is a finitely generated abelian group and, therefore, it decomposes as $$A(K) \cong A(K)_{\text{tors}} \oplus \mathbb{Z}^r,$$ where $A(K)_{\text{tors}}$ is a finite subgroup called the torsion subgroup and $r$ is a non-negative integer. The problem of producing bounds for the torsion subgroup $A(K)_{\text{tors}}$ has a very long and rich history. When $g=1$ and $d \leq 2$, i.e., when $K$ is either $\mathbb{Q}$ or a quadratic number field and $A/K$ is an elliptic curve defined over $K$, we have a complete classification of the possible torsion subgroups that can occur (see \cite{maz}, \cite{kamienny}, and \cite{kenkumomose}). For higher degree $d$, the Uniform Boundedness Theorem, due to Merel \cite{merel1996}, tells us that for every $d>0$ there exists a bound $B(d)$ such that for every number field $K$ of degree $d$ and every elliptic curve $E/K$ we have that $$|E(K)_{\text{tors}}| < B(d).$$ For higher dimensional abelian varieties, a lot less is known. For example, even for abelian surfaces $A/\mathbb{Q}$ there is no known uniform bound for the group $A(\mathbb{Q})_{\text{tors}}$.

Given an abelian variety $A/K$, it is natural to ask how the torsion subgroup $A(K)_{\text{tors}}$ changes upon quadratic twisting. More specifically, let $K(\sqrt{d})$ be a quadratic extension of $K$, for some element $d$ of $K$, and denote by $A^d/K$ the quadratic twist of $A/K$ by $d$ (See Example \ref{examplequadratictwists} below for the definition). One can then ask the following.
\begin{question}\label{question1}
If $A(K)_{\text{tors}} \neq A^d(K)_{\text{tors}}$, then are there any properties (depending on $A/K$) that the extension $K(\sqrt{d})/K$ must necessarily have? 
\end{question}
\noindent We also refer the reader to \cite{silverman1984}, which contains related results, including the statement that over a number field, an abelian variety has only finitely many twists that admit a $K$-rational torsion point of order strictly greater than $2$.

In this article, we will primarily be interested in the relationship between the primes of $K$ that ramify in $K(\sqrt{d})$ and the primes of good reduction of $A/K$. In fact, we are able to prove a more general theorem concerning twists of $A/K$ by irreducible rational representations, assuming certain ramification conditions on $K$. Our methods are local relying on the interplay between torsion points and reduction properties of abelian varieties (see \cite{cx} and \cite{mentzeloskodairandtorsion} for more on this relationship).

Let $K$ be a field, let $L/K$ be a finite Galois extension with Galois group denoted by $G$, and let $A/K$ be an abelian variety. The group algebra $\mathbb{Q}[G]$ decomposes as a direct sum $$\mathbb{Q}[G] = \bigoplus_{\rho} \mathbb{Q}[G]_{\rho},$$ where the direct sum is indexed by the irreducible rational representations of $G$ and $\mathbb{Q}[G]_{\rho}$ is the $\rho$-isotypic component of $\mathbb{Q}[G]$. If $$I_{\rho}= \mathbb{Q}[G]_{\rho} \cap \mathbb{Z}[G],$$ then we can construct (see Definition \ref{definitiontwist} below) an abelian variety defined over $K$ denoted by $$A_{\rho} :=I_{\rho} \otimes_{\mathbb{Z}} A$$ and called the $\rho$-twist of $A/K$. Such twists were considered by Mazur and Rubin \cite{mazurrubin07}, \cite{mazurrubin10} in order to study Selmer ranks of elliptic curves. We will not consider Selmer ranks in this article, but we expect that our results will have applications to Selmer groups of quadratic twists of abelian varieties.

Our first result concerns the possible orders of $K$-rational torsion points of $A_{\rho}$, under certain assumptions on $K$ and $A/K$.

\begin{theorem}\label{theoremtwistgeneralform} Let $K$ be a local field of characteristic $0$ with valuation $v_K$ and residue field $k$ of characteristic $p >2$. Let $A/K$ be a simple $g$-dimensional abelian variety that has good reduction and let $L/K$ be a totally ramified finite Galois extension of degree $m \geq 2$ with Galois group $G$. Then
\begin{enumerate}
    \item If $\rho$ is a non-trivial irreducible rational representation of $G$ and $v_K(p)<\frac{p-1}{m}$, then the twist $A_{\rho}/K$ cannot have any $K$-rational points of order $p$.
    \item If $\rho$ is a non-trivial irreducible rational representation of $G$ and the twist $A_{\rho}/K$ has a $K$-rational point of prime order $\ell$ with $\ell \neq p$, then $\ell \leq 2g+1$.
\end{enumerate}
\end{theorem}

\begin{remark}
Example \ref{example3torsion} below shows that the assumption that $v_K(p)<\frac{p-1}{m}$ in Part $(i)$ of Theorem \ref{theoremtwistgeneralform} is necessary and cannot be removed. The assumption that $A/K$ has good reduction in Theorem \ref{theoremtwistgeneralform} is also necessary and cannot be removed. The latter is, of course, expected because every elliptic curve over $\mathbb{Q}$ is a quadratic twist of another elliptic curve defined over $\mathbb{Q}$.
\end{remark} 

 As an application of Theorem \ref{theoremtwistgeneralform} we obtain the following theorem.

\begin{theorem}\label{theoremexample}
Let $A/\mathbb{Q}$ be a simple $g$-dimensional abelian variety of conductor $N$ and let $d$ be a square-free integer. If $d$ has a prime divisor $p>3$ with $p \nmid N$, then the quadratic twist $A^d/\mathbb{Q}$ cannot have any $\mathbb{Q}$-rational points of prime order $\ell$ with $\ell>3$.
\end{theorem}

\begin{remark}
    The above theorem partially generalizes \cite[Theorem 1.2, Part $(i)$]{mentzelosagasheconjecture} and can also be thought of as a higher dimensional analog of a proposition of Mazur and Gouv\^ea \cite[Proposition 1]{gouveamazur}. Moreover, as is shown in \cite[Remark 1.4]{mentzelosagasheconjecture} if $d = - 1$, then the conclusion of Theorem \ref{theoremexample} does not hold.
\end{remark}

We now explain some applications of Theorem \ref{theoremtwistgeneralform} to torsion growth of abelian varieties over quadratic extensions. Let $A/K$ be an abelian variety over a number field $K$ and let $L/K$ be a quadratic extension. If  $A(L)_{tors} \setminus A(K)_{tors} $ contains a point of prime order $p$, then it follows from the N\'eron-Ogg-Shafarevich Criterion \cite[Theorem 1]{st} that the primes of $K$ that ramify in $L$ are contained in the set
$$\{ \mathfrak{p} \text{ is a prime of } \mathcal{O}_K \; : \; \mathfrak{p} \text{ lies above } p  \; \text{or } A/K \text{ has bad reduction at } \mathfrak{p} \}.  $$ The following theorem tells us that, under a certain assumption on the field $K$, every prime which ramifies in $L$ and lies above $p$ must be a prime of bad reduction of $A/K$.

\begin{theorem}\label{theoremramification}
Let $A/K$ be a simple abelian variety over a number field $K$ and let $L$ be a quadratic extension of $K$. Assume that $A(L)_{tors} \setminus A(K)_{tors} $ contains a point of prime order $p$. If for all primes $\mathfrak{p}$ of $\mathcal{O}_K$ with associated valuation $v_{\mathfrak{p}}$ we have that $v_{\mathfrak{p}}(p) < \frac{p-1}{2}$, then the primes of $K$ that ramify in $L$ are contained in 
$$\{ \mathfrak{p} \text{ is a prime of } \mathcal{O}_K \; : A/K \text{ has bad reduction at } \mathfrak{p} \}.  $$

\end{theorem}

We note that Theorem \ref{theoremramification} is sharp in the sense that the condition $v_{\mathfrak{p}}(p) < \frac{p-1}{2}$ is optimal (see Example \ref{examplesharp} below).

Finally, we turn our attention to elliptic curves $E/\mathbb{Q}$ and quadratic extensions $L$ that satisfy the Heegner hypothesis for $E/\mathbb{Q}$, i.e., $L$ is an imaginary quadratic field and all prime divisors of the conductor $N$ of $E/\mathbb{Q}$ split in $L$. Under this assumption, we show that if $L \neq \mathbb{Q}(\sqrt{-3})$, then the quotient $|E_L(L)_{\mathrm{tors}}|/ |E(\mathbb{Q})_{\mathrm{tors}}|$ is a power of $2$ (see Corollary \ref{corollaryheegner} below).

\begin{acknowledgement}
The author was supported by Czech Science Foundation (GA\v CR) grant 21-00420M and by Charles University Research Center program No.UNCE/SCI/022. I would like to thank Dino Lorenzini for providing some useful comments and suggestions on an earlier version of this manuscript. I would also like to thank two anonymous referees for many valuable comments and suggestions.
\end{acknowledgement}

\section{Twists of abelian varieties}\label{sectiontwistsofabelianvarieties}

In this section, we prove Theorem \ref{theoremtwistgeneralform} and then derive some corollaries on torsion points of quadratic twists. Before we begin our proof, we recall some background material on twists of abelian varieties. We refer the reader to \cite{mazurrubinsilverberg}, \cite{silverberg08}, and \cite{milneonthearithmeticofabelianvarieties} for further information concerning this topic.

Let $K$ be any field, let $A/K$ be an abelian variety, and let $L/K$ be a finite Galois extension. If $X/K$ is any abelian variety, then we will denote by $X_L/L$ the base change of $X/K$ to $L$. An $L/K$-form of the abelian variety $A/K$ is a pair $(B, \psi)$, where $B/K$ is an abelian variety and $\psi:  A_L \longrightarrow B_L$ is an isomorphism which is defined over $L$. Two $L/K$-forms $(B,\psi)$ and $(B', \psi')$ are called equivalent if they are isomorphic over $K$.

If $(B, \psi)$ is an $L/K$-form of $A/K$, then the class of the map $\xi: \text{Gal}(L/K) \longrightarrow  \text{Aut}_{L}(A_L)$ given by $\xi(\sigma)= \psi^{-1} \circ \psi^{\sigma}$ for all $\sigma \in \text{Gal}(L/K)$ is an element of $ \text{H}^1(\text{Gal}(L/K), \text{Aut}_{L}(A_L))$, i.e., $\xi$ is a $1$-cocycle for $\text{Gal}(L/K)$ with values in $\text{Aut}_{L}(A_L)$. This association induces a bijection between the set of $L/K$-forms modulo equivalence and the pointed set $ \text{H}^1(\text{Gal}(L/K), \text{Aut}_{L}(A_L))$ (see \cite[Chapter III]{serregaloiscohomology}).

 To ease notation in what follows we shall write $\text{Aut}_{L}(A^n)$ for $\text{Aut}_{L}((A_L)^n)$ for every positive integer $n$.

\begin{definition}\label{definitiontwist}
Let $A/K$ be an abelian variety and let $R$ be a commutative ring with a ring homomorphism $R \longrightarrow \text{End}_K(A)$. Fix a positive integer $n$. Let $I$ be a finitely generated free $R$-module of rank $n$ with a continuous right action of the absolute Galois group of $K$ and fix an $R$-module isomorphism $\psi : R^n \longrightarrow I$. If GL$_n(R)$ is regarded as a subgroup of $\text{Aut}_{L}(A^n)$, then the homorphism $\xi: G \longrightarrow \text{Aut}_{L}(A^n)$ given by  $\xi(\sigma)=\psi^{-1} \circ \psi^{\sigma}$ is a $1$-cocycle for $\text{Gal}(L/K)$, i.e., the class of $\xi$ belongs to $\text{H}^1(\text{Gal}(L/K), \text{Aut}_{L}(A^n))$. Therefore, there exists an abelian variety $(I \otimes_{R} A)/K$ corresponding to $\xi$. We call the abelian variety $(I \otimes_{R} A)/K$ the twist of $A/K$ by $I$. 
\end{definition}

Note that according to Definition \ref{definitiontwist} the variety $(I \otimes_{R} A)/K$ is an $L/K$-form of $A^n/K$. Moreover, it is shown in \cite[Remark 1.2]{mazurrubinsilverberg} that the definition of the twist is independent of the choice of $\psi$. We proceed with two important examples that will be used below.

\begin{example}(Weil restriction)
Let $L/K$ be a finite extension and let $B/L$ be an abelian variety. Recall that the Weil restriction Res$_{L/K}(B)$ of $B/L$ from $L$ to $K$, is the scheme defined over $K$ representing the functor from $K$-schemes to sets given by $S \mapsto B(S \times_{K} L)$.

Let $L/K$ be a finite Galois extension with $G=\text{Gal}(L/K)$ and let $I=R[G]$. Then the twist $I \otimes_R B$ is isomorphic to $\text{Res}_{L/K}(B)$ (see \cite[Proposition 4.1]{mazurrubinsilverberg}).

\end{example}

\begin{example}\label{examplequadratictwists}(Quadratic twists, see also \cite[Section 3.1.3]{silverberg08})
Let $A/K$ be an abelian variety over a field $K$ and let $L=K(\sqrt{d})$ be a quadratic extension. Let $\chi_L : G_K \longrightarrow \{ \pm 1 \}$ be the unique nontrivial quadratic character of the absolute Galois group $G_K$ of $K$ that factors through $\text{Gal}(L/K)$. Let $R=\mathbb{Z}$ considered as a subset of $\text{End}_K(A)$ via the identification of $R$ with $\{ [m] : \; m \in \mathbb{Z} \}$, where $[m]$ in the multiplication by $m$ on $A/K$. Consider $I$ to be a free rank $1$ module over $\mathbb{Z}$ equipped by an action of $G_K$ given by $i^{\sigma}=\chi_L(\sigma) \cdot i$ for $i \in I$ and $\sigma \in G_K$. Denote by $A^d/K$ the abelian variety $(I \otimes_R A)/K$. Here the map $\psi$ is defined as follows; fix a generator $i_0$ of $I$ and let $\psi: R \longrightarrow I$ be the isomorphism given by  $\psi(m)=m \cdot i_0$. The associated $1$-cocycle is $\xi: \text{Gal}(L/K) \longrightarrow  \text{Aut}_{L}(A)$ given by $\xi(\gamma)=[-1]$ and $\xi(\text{id})=[1]=\text{id}$, where $\gamma$ is the generator of $\text{Gal}(L/K)$. The abelian variety $A^d/K$ is called the quadratic twist of $A/K$ by $d$. 

In the special case where $E/K$ is an elliptic curve and char$(K) \neq 2,3$  using \cite[Example X.2.4]{aec} we see that $E^{d}/K$ is the usual quadratic twist of $E/K$ by $d$, i.e., if $E/K$ is given by a short Weierstrass equation $$y^2=x^3+ax+b,$$ for some some $a,b \in K$, then $E^d/K$ is given by a short Weierstrass equation $$dy^2=x^3+ax+b.$$
\end{example}

Consider now a finite Galois extension $L/K$ with Galois group $G:=\text{Gal}(L/K)$. The group algebra $\mathbb{Q}[G]$ decomposes as a direct sum $\mathbb{Q}[G] = \bigoplus_{\rho} \mathbb{Q}[G]_{\rho}$, where the direct sum is indexed by the irreducible rational representations of $G$ and $\mathbb{Q}[G]_{\rho}$ is the $\rho$-isotypic component of $\mathbb{Q}[G]$.

\begin{definition}
     Let $A/K$ be an abelian variety. If $\rho$ is an irreducible rational representation of the group $G$, then let $$I_{\rho}= \mathbb{Q}[G]_{\rho} \cap \mathbb{Z}[G].$$ 
     Define the $\rho$-twist of $A/K$ by $$A_{\rho} :=I_{\rho} \otimes_{\mathbb{Z}} A.$$
\end{definition}  

The following theorem, due to Mazur, Rubin, and Silverberg, will be very useful in the proof of Theorem \ref{theoremtwistgeneralform}.

\begin{theorem}(see \cite[Theorem 4.5]{mazurrubinsilverberg})\label{weilrestrictionandtwists}
Let $L/K$ be a finite Galois extension with Galois group $G$ and let $A/K$ be an abelian variety. Then $\text{Res}_{L/K}(A_L)$ is isogenous over $K$ to $\prod_{\rho} A_{\rho}$, where the product is taken over all irreducible rational representations of $G$.
\end{theorem}

Before we proceed to the proof of Theorem \ref{theoremtwistgeneralform}, we need to briefly recall a few basic facts concerning reduction of abelian varieties. The interested reader can find more information on this topic in \cite{neronmodelsbook} and \cite{neronmodelslorenzini}. Let $K$ be a local field, i.e., $K$ is a complete field with respect to a discrete valuation $v_K$ and has finite residue field $k$, and let $A/K$ be an abelian variety of dimension $g$. We denote by $\mathcal{A}/\mathcal{O}_K$ the N\'eron model of $A/K$. The special fiber $\mathcal{A}_k/k$ of $\mathcal{A}/\mathcal{O}_K$ is a smooth commutative group scheme. We denote by $\mathcal{A}^0_k/k$ the connected component of the identity of $\mathcal{A}_k/k$. By a theorem of Chevalley (see \cite[Theorem $1.1$]{con}) there exists a short exact sequence $$0\longrightarrow T \times U \longrightarrow \mathcal{A}^0_k \longrightarrow B \longrightarrow 0, $$
where $T/k$ is a torus, $U/k$ is a unipotent group, and $B/k$ is an abelian variety. The number $\text{dim}(U)$ (resp. $\text{dim}(T)$, $\text{dim}(B)$) is called the unipotent (resp. toric, abelian) rank of $A/K$. By construction we have the following equality $g=\text{dim}(U)+\text{dim}(T)+\text{dim}(B)$. We say that $A/K$ has purely additive reduction if $g=\text{dim}(U)$, or equivalently, if $\text{dim}(T)=\text{dim}(B)=0$.

\begin{proof}[Proof of Theorem \ref{theoremtwistgeneralform}]

To ease notation let $\rho_0,..., \rho_n $ be the irreducible rational representations of $G$ with $\rho_0$ the trivial representation (corresponding to the trivial twist of $A/K$) and let $A_i:=A_{\rho_i}$ for $i=1,...,n$. 

\begin{lemma}\label{claimreduction}
For every $i=1,...,n$ the abelian variety $A_i$ has purely additive reduction.
\end{lemma}
\begin{proof}[Proof Lemma \ref{claimreduction}]
Let $W:=$Res$_{L/K}(A_L)$ be the Weil restriction of the base change $A_L/L$ of $A/K$ to $L$. Theorem \ref{weilrestrictionandtwists} tells us that there exists an isogeny $W \longrightarrow A \times A_1 \times ... \times A_{n}$ defined over $K$. Since $L/K$ is totally ramified and $A/K$ has good reduction, using \cite[Remark on Page 179]{milneonthearithmeticofabelianvarieties} we obtain that the abelian rank of $W/K$ is equal to the dimension of $A/K$ and that the toric rank of $W/K$ is zero. Recall that we assume that $A/K$ has good reduction. This implies that the abelian rank of $A/K$ is equal to the dimension of $A/K$. Denote by $\mathcal{A}/\mathcal{O}_K$ the N\'eron model of $A/K$ and by $\mathcal{A}_i/\mathcal{O}_K$ the N\'eron model of $A_i/K$, for every $i=1,...,n$. Since the N\'eron model of $(A \times A_1 \times ... \times A_n )/ K$ is $(\mathcal{A} \times \mathcal{A}_1 \times ... \times \mathcal{A}_{n}) / \mathcal{O}_K$ and the abelian, unipotent, and toric ranks of abelian varieties are preserved under isogeny (see \cite[Corollaire IX.2.2.7]{sga7}), we find that $A_i/K$ has purely additive reduction for every $i=1,...,n$. This proves our lemma.
\end{proof}

 {\it Proof of Part $(i)$:}
Assume that for some $j>0$ the twist $A_j/K$ has a $K$-rational point of order $p$ and we will arrive at a contradiction. Since $A_j/K$ acquires good reduction in $L$ (because it is an $L/K$-form of $A/K$) and $v_K(p)<\frac{p-1}{m}$, using \cite[Theorem 1.1]{mentzeloskodairandtorsion}, applied to the base extension of $A_j/K$ to the maximal unramified extension of $K$, we find that $A_j/K$ cannot have purely additive reduction. However, this contradicts Lemma \ref{claimreduction}. This proves part $(i)$.

{\it Proof of Part $(ii)$:}
Assume that for some $j>0$ the twist $A_j/K$ has a $K$-rational point of order $\ell$ with $\ell \neq p$. We need to show that $\ell \leq 2g+1$, where $g$ is the dimension of $A/K$. Using Lemma \ref{claimreduction} we find that $A_j/K$ has purely additive reduction. Therefore, using \cite[Main Theorem Part (iii)]{cx} we find that $\ell \leq 2g+1$. This completes the proof of Theorem \ref{theoremtwistgeneralform}.

\end{proof}

\begin{emptyremark}\label{remarkquadratictwist}
If $L/K$ is a finite abelian Galois extension, then the irreducible rational representations of $\text{Gal}(L/K)$ are in one-to-one correspondence with the cyclic sub-extensions of $K$ in $L$. Assume now that $L/K$ is a quadratic extension with Galois group $G$ and denote by $\sigma$ the generator of $G$. According to \cite[Section 5.3]{silverberg08}, if $\rho_L$ is the irreducible rational representation corresponding to the extension $L/K$, then $\mathbb{Z}[G]_{\rho_L}=(\sigma-1)\mathbb{Z}$ is a free rank one $\mathbb{Z}$-module. Moreover, if $\chi_L : G_K \longrightarrow \{ \pm 1 \}$ is the unique nontrivial quadratic character of the absolute Galois group $G_K$ of $K$ that factors through $\text{Gal}(L/K)$, then every $\gamma \in G_K$ acts on $\mathbb{Z}[G]_{\rho_L}$ as multiplication by $\chi_L(\gamma)$. Therefore, it follows that the corresponding abelian variety $A_{\rho_L}/K$ is just the quadratic twist $A^d/K$ of Example \ref{examplequadratictwists}, for some $d\in L$ with $L=K(\sqrt{d})$.
\end{emptyremark}

\begin{example}\label{exampletwist} (See also the \texttt{MathOverflow} post \cite{mathoverflowpost}) In this example we show that it is possible for both an elliptic curve and a quadratic twist of it to have rational points of odd orders. Let $E/\mathbb{Q}$ be the elliptic curve given by the following Weierstrass equation $$y^2+xy+y=x^3-76x+298.$$
The curve $E/\mathbb{Q}$ has LMFDB \cite{lmfdb} label 50a2 and $E(\mathbb{Q}) \cong \mathbb{Z}/3\mathbb{Z}$. Therefore, we see that $E(\mathbb{Q}_5)$ contains a point of order $3$. Using SAGE \cite{sagemath} we find that the quadratic twist of $E/\mathbb{Q}$ by $5$, which we denote by $E^5/\mathbb{Q}$, has Weierstrass equation $$y^2+xy+y=x^3+x^2-3x+1$$
and LMFDB \cite{lmfdb} label 50b3. Moreover, we have that $E^5(\mathbb{Q}) \cong \mathbb{Z}/5\mathbb{Z}$. Therefore, we see that $E^5(\mathbb{Q}_5)$ contains a point of order $5$. Thus, both the curve $E/\mathbb{Q}_5$ and the curve $E^5/\mathbb{Q}_5$ have (nontrivial) $\mathbb{Q}_5$-rational points of finite order.
\end{example}

\begin{example}\label{example3torsion}
This example shows that the condition $v_K(p)<\frac{p-1}{m}$ in Theorem \ref{theoremtwistgeneralform} is necessary, with $p=3$ and $m=2$. Consider the elliptic curve $E/\mathbb{Q}$ given by the following Weierstrass equation $$y^2+y=x^3+x^2-9x-15.$$ The curve $E/\mathbb{Q}$ has LMFDB \cite{lmfdb} label 19a2. Therefore, the base change $E_{\mathbb{Q}_3}/\mathbb{Q}_3$ has good reduction. Using SAGE \cite{sagemath} we find that the quadratic twist of $E/\mathbb{Q}$ by $-3$, which we denote by $E^{-3}/\mathbb{Q}$, has Weierstrass equation $$y^2+y=x^3-84x+315$$
and LMFDB \cite{lmfdb} label 171b2. Moreover, we have that $E^{-3}(\mathbb{Q}) \cong \mathbb{Z} \times \mathbb{Z}/3\mathbb{Z}$. This implies that the curve $E^{-3}_{\mathbb{Q}_3}/\mathbb{Q}_3$ has a $\mathbb{Q}_3$-rational point of order $3$.
\end{example}

\begin{corollary}\label{corollaryquadratictwist}
Let $K$ be a number field, let $A/K$ be a simple abelian variety, and let $L=K(\sqrt{d})$ be a quadratic extension of $K$, for some $d \in K$. Fix a rational prime $p$. Let $\mathfrak{P}$ be a prime of $K$ which lies above $p$ and denote by $v_{\mathfrak{P}}$ the corresponding valuation of $K$. Assume that $\mathfrak{P}$ ramifies in $L$ and that $A/K$ has good reduction modulo $\mathfrak{P}$. If $v_{\mathfrak{P}}(p) < \frac{p-1}{2}$, then the quadratic twist $A^d/K$ cannot have a $K$-rational point of order $p$.
\end{corollary}

\begin{proof}
Since $L/K$ is a quadratic extension and we assume that $\mathfrak{P}$ ramifies in $L$, we know that there exists a unique prime $\mathfrak{P}'$ of $L$ which lies above $\mathfrak{P}$. Let $K_{\mathfrak{P}}$ be the completion of $K$ at $\mathfrak{P}$ and let $L_{\mathfrak{P}'}$ be the completion of $L$ at $\mathfrak{P}'$. It follows from basic algebraic number theory (see \cite[Chapter II]{neukirchbook}) that under our assumptions, the extension $L_{\mathfrak{P}'}/K_{\mathfrak{P}}$ is a (totally) ramified quadratic extension of local fields. Consider the base change $A_{K_{\mathfrak{P}}}/K_{\mathfrak{P}}$ of $A/K$ to $K_{\mathfrak{P}}$. Using Paragraph \ref{remarkquadratictwist} we see that Part $(i)$ of Theorem \ref{theoremtwistgeneralform} implies that the abelian variety $A^d_{K_{\mathfrak{P}}}/K_{\mathfrak{P}}$ cannot have a $K_{\mathfrak{P}}$-rational point of order $p$. Therefore, we find that $A^d/K$ cannot have a $K$-rational point of order $p$. This proves our corollary.
\end{proof}

As an application of the previous corollary, we will now prove Theorem \ref{theoremexample} of the introduction. Before we proceed to our proof we need to recall the following useful lemma.

\begin{lemma}
Let $K$ be a field with char$(K) \neq 2$ and let $A/K$ be a simple abelian variety. Let $L/K$ be an abelian Galois extension of degree $m>1$ and assume that the exponent of the Galois group Gal$(L/K)$ is $2$. Then there exist $d_i \in K$ for $i=1,...,m$ with associated quadratic twists $A_i:=A^{d_i}$ and a group homomorphism $$\displaystyle\bigoplus_{i=1}^m A_i(K) \longrightarrow A_L(L)$$ whose kernel and cokernel are annihilated by $m$.
\end{lemma}
\begin{proof}
Using \cite[Page 165]{laskalorenz} we find that this is a special case of \cite[Lemma 1.1]{laskalorenz} with $e=2$.
\end{proof}

When $m=2$, which is the case of primary interest for us, we immediately obtain the following corollary.

\begin{corollary}\label{corollarytwist}
Let $L/K$ be a quadratic extension of number fields with $L=K(\sqrt{d})$ and let $A/K$ be a simple abelian variety. If $n$ is odd, then $$A(K)[n] \oplus A^{d}(K)[n] \cong A_L(L)[n].$$
\end{corollary}

We are now ready to proceed with our proof.

\begin{proof}[Proof of Theorem \ref{theoremexample}] 
    Assume that the twist $A^d/\mathbb{Q}$ has a $\mathbb{Q}$-rational point of prime order $\ell$ for some $\ell > 3$ and we will find a contradiction. Recall that we assume the existence of a prime $p>3$ that divides $d$ but it does not divide $N$.

    Assume first that $\ell \neq p$. Since $p$ and $N$ are coprime, we see that $A/\mathbb{Q}$ has good reduction modulo $p$. Fix an algebraic closure $\overline{\mathbb{Q}}$ of $\mathbb{Q}$. Let $K_{\ell}=\mathbb{Q}(A[\ell])$ be the $\ell$-division field of $A/K$, i.e., the minimal field of definition of all the $\ell$-torsion points of $A(\overline{\mathbb{Q}})$. If $L=\mathbb{Q}(\sqrt{d})$, then it follows from Corollary \ref{corollarytwist} that $$A(\mathbb{Q})[\ell] \oplus A^{d}(\mathbb{Q})[\ell] \cong A_L(L)[\ell].$$ Since we assume that $A^d/\mathbb{Q}$ has a $\mathbb{Q}$-rational point of order $\ell$, we find that $A_L/L$ has an $L$-rational point of order $\ell$. Therefore, we have that $L \subseteq K_{\ell}$. On the other hand, since $A/\mathbb{Q}$ has good reduction modulo $p$, using the N\'eron-Ogg-Shafarevich Criterion \cite[Theorem 1]{st}, we find that $p$ is unramified in $K_{\ell}$. However, this contradicts the fact that $L \subseteq K_{\ell}$ because $p$ divides $d$ and, hence, it ramifies in $L$.

    Assume now that every prime divisor of $d$ divides $\ell$, i.e., that $d = \pm \ell$. Since $\ell \geq 5$, we see that if $v_{\ell}$ is the valuation corresponding to $\ell$, then $v_{\ell}(\ell) =1 < \frac{\ell-1}{2}$. Therefore, applying Corollary \ref{corollaryquadratictwist} (for $K=\mathbb{Q}$ and $\ell=\mathfrak{P}$) we find that $A^d/\mathbb{Q}$ cannot have a $\mathbb{Q}$-rational point of order $\ell$, which is again a contradiction. This proves our theorem.
\end{proof}

\begin{remark}
Let $R$ be a complete discrete valuation ring with fraction field $K$ of characteristic $p>0$ and algebraically closed residue field. One may wonder whether there can exist an analog of Theorem \ref{theoremtwistgeneralform} over the field $K$. Unfortunately, such an analog does not seem to exist, as we now explain. Let $E/K$ be an elliptic curve with good reduction and let $\mathcal{E}/R$ be the Néron Model of $E/K$. Assume that the Hasse invariant of $\mathcal{E}/R$ has vanishing order $\frac{p-1}{2}$ (see \cite[Section 5]{liedtkeschroer} for information on the Hasse invariant of $\mathcal{E}/R$). Let $E^{(p)}/K$ be the Frobenius pullback of the curve $E/K$, which still has good reduction by \cite[Proposition 7.2]{liedtkeschroer}. According to \cite[Proposition 8.3]{liedtkeschroer} for every separable quadratic extension $K'/K$ the corresponding quadratic twist $\widetilde{E^{(p)}}/K$ of $E^{(p)}/K$ has a $K$-rational point of order $p$.
\end{remark}

\section{Torsion growth of abelian varieties and an application to a conjecture of Gross and Zagier}\label{sectiontorsiongrowthgenralcase}

In this section, we consider torsion growth questions for abelian varieties over quadratic fields. After proving Theorem \ref{theoremramification}, we turn to elliptic curves $E/\mathbb{Q}$ and we study the possible torsion growth over quadratic extensions satisfying the Heegner hypothesis with respect to $E/\mathbb{Q}$.

\begin{proof}[Proof of Theorem \ref{theoremramification}]
    Let $\mathfrak{p}$ be a prime of $\mathcal{O}_K$ which ramifies in $L$ and lies above $p$. Assume that $A/K$ has good reduction modulo $\mathfrak{p}$ and we will find a contradiction. Write $L=K(\sqrt{d})$ for some $d \in K$ and denote by $v_{\mathfrak{p}}$ the valuation associated to $\mathfrak{p}$. Since we assume that $v_{\mathfrak{p}}(p) < \frac{p-1}{2}$, using Corollary \ref{corollaryquadratictwist} we find that the twist $A^d/K$ cannot have a $K$-rational point of order $p$.  However, this is a contradiction because by Corollary \ref{corollarytwist} we have that $$A(K)[p] \oplus A^{d}(K)[p] \cong A_L(L)[p]$$ and we assume that $A(L)_{tors} \setminus A(K)_{tors} $ contains a point of order $p$. This proves our theorem.
\end{proof}

\begin{example}\label{examplesharp}
  This example shows that Theorem \ref{theoremramification} is sharp for $p:=7$. Consider the elliptic curve $E/\mathbb{Q}$ with LMFDB \cite{lmfdb} label 26.b1. This elliptic curve is given by the following Weierstrass equation $$y^2+xy+y=x^3-x^2-213x-1257.$$ Let $\mathbb{Q}(\zeta_7)^+$ be the maximal totally real subfield of the cyclotomic field $\mathbb{Q}(\zeta_7)$. Using the LMFDB database we find that $E_{\mathbb{Q}(\zeta_7)^+}(\mathbb{Q}(\zeta_7)^+)_{tors}$ is trivial and that $E_{\mathbb{Q}(\zeta_7)}(\mathbb{Q}(\zeta_7))_{tors} \cong \mathbb{Z}/7\mathbb{Z}$. Thus the curve $E_{\mathbb{Q}(\zeta_7)^+}/\mathbb{Q}(\zeta_7)^+$ acquires a torsion point of order $7$ over the quadratic extension $\mathbb{Q}(\zeta_7)/\mathbb{Q}(\zeta_7)^+$. Moreover, if $\mathfrak{p}^+$ is the prime of $\mathbb{Q}(\zeta_7)^+$ that lies above $7$, then the curve $E_{\mathbb{Q}(\zeta_7)^+}/\mathbb{Q}(\zeta_7)^+$ has good reduction modulo $\mathfrak{p}^+$. Finally, if $v_{\mathfrak{p}^+}$ is the associated valuation, we see that $v_{\mathfrak{p}^+}(7)=\frac{7-1}{2}$.
\end{example}

\begin{example}
    The following is an interesting example where an elliptic curve acquires both a torsion point of large order and everywhere good reduction over a quadratic extension. Consider the elliptic curve $E/\mathbb{Q}$ with LMFDB \cite{lmfdb} label 1225.b2. This elliptic curve is given by the following Weierstrass equation $$y^2+xy+y=x^3+x^2-8x+6.$$  Let $\mathbb{Q}(\zeta_{35})^+$ be the maximal totally real subfield of the cyclotomic field $\mathbb{Q}(\zeta_{35})$. Recall that the degree of the extension $\mathbb{Q}(\zeta_{35})^+/\mathbb{Q}$ is $12$. The field $\mathbb{Q}(\zeta_{35})^+$ contains a sextic number field $K$ with defining polynomial $$x^6-x^5-7x^4+2x^3+7x^2-2x-1.$$ Using SAGE \cite{sagemath} (or MAGMA) it is easy to see that the torsion subgroup $E_K(K)_{tors}$ is trivial and that $E_K/K$ has bad reduction modulo two prime ideals. In fact, $E_K/K$ is a twist of an elliptic curve that appears in \cite[Section 7.2]{lorenzinitorsionandexceptionalunits}. On the other hand, using LMFDB we find that the curve $E_{\mathbb{Q}(\zeta_{35})^+}/\mathbb{Q}(\zeta_{35})^+$ has a torsion point of order $37$ and everywhere good reduction. 
\end{example}

    
    For an elliptic curve $E/\mathbb{Q}$ and quadratic field $L$ the relationship between $E(\mathbb{Q})_{\text{tors}}$ and $E_L(L)_{\text{tors}}$ has been studied by Gonz\'alez-Jim\'enez and Tornero in \cite{gonzalezjimeneztornero} and \cite{gonzalezjimeneztorneroII}. Moreover, Najman, answering a problem posed by Gonz\'alez-Jim\'enez and Tornero, in \cite{najmantwists} gave sharp upper bounds on the number of quadratic extensions for which $E(\mathbb{Q})_{\text{tors}} \subsetneq E_L(L)_{\text{tors}}$. In the following theorem, we show that if every prime of bad reduction for $E/\mathbb{Q}$ is unramified in $L$, then the possible torsion growth is very restricted.

\begin{theorem}\label{theoremtorsiongrowth}
Let $E/\mathbb{Q}$ be an elliptic curve, let $L/\mathbb{Q}$ be a quadratic extension, and assume that every prime of bad reduction of $E/\mathbb{Q}$ is unramified in $L$. Then

\begin{enumerate}
     \item The set $E_L(L)_{\mathrm{tors}}\setminus E(\mathbb{Q})_{\mathrm{tors}}$ cannot contain points of prime order $p>3$.
     \item If there exists a prime $p \neq 3$ which ramifies in $L$, then the quotient $|E_L(L)_{\mathrm{tors}}|/ |E(\mathbb{Q})_{\mathrm{tors}}|$ is equal to a power of $2$.
\end{enumerate}

\end{theorem}

\begin{proof}
 
 {\it Proof of Part $(i)$:}  This part follows directly from Theorem \ref{theoremramification}.
 
 {\it Proof of Part $(ii)$:} Using the previous part we see that we only need to show that $E_L(L)_{\mathrm{tors}} \setminus E(\mathbb{Q})_{\mathrm{tors}}$ cannot contain a point of order $3$. Assume that $E_L(L)_{\mathrm{tors}} \setminus E(\mathbb{Q})_{\mathrm{tors}}$ contains a point $P$ of order $3$ and we will find a contradiction.  Write $L=\mathbb{Q}(\sqrt{d})$ with $d \in \mathbb{Z}$ square-free. By Corollary \ref{corollarytwist} we have that $E_L(L)[3] \cong  E(\mathbb{Q})[3] \oplus E^{d}(\mathbb{Q})[3]$, which implies that there exists a point $P' \in E^{d}(\mathbb{Q})[3]$ of order $3$. Since $p$ is a prime of good reduction for $E/\mathbb{Q}$ and $p$ divides $d$, we find that $E^d/\mathbb{Q}$ has reduction of Kodaira type I$_0^*$, II, or I$_8^*$ modulo $p$ by \cite[Proposition 1]{com} and \cite[Table II]{com} (see also
the hypotheses for this table at the bottom of page 58 of \cite{com}). However, since $p\neq 3$, using \cite[Proposition 2.4]{mentzelosagasheconjecture}, we find that rational elliptic curves with a point of order $3$ cannot have reduction of Kodaira type I$_0^*$, II, or I$_8^*$ modulo $p$. This implies that $E^d/\mathbb{Q}$ cannot have a $\mathbb{Q}$-rational point of order $3$, which is a contradiction. This proves our theorem.

\end{proof}

\begin{example}
    This example shows that the assumption that $p\neq 3$ in Part $(ii)$ of Theorem \ref{theoremtorsiongrowth} is necessary, and cannot be removed. Consider the elliptic curve $E/\mathbb{Q}$ given by the following Weierstrass equation 
    $$y^2+xy+y=x^3+549x+2202.$$
    This is the curve with LMFDB \cite{lmfdb} label 50a4 (Cremona \cite{cremonabook} label 50a4). It has bad reduction modulo $2$ and modulo $5$, and good reduction modulo every other prime. Using the LMFDB database we find that on the one hand $E(\mathbb{Q})_{\mathrm{tors}}$ is trivial while on the other hand $E_{\mathbb{Q}(\sqrt{-3})}(\mathbb{Q}(\sqrt{-3}))_{\mathrm{tors}} \cong \mathbb{Z}/3\mathbb{Z}$.
\end{example}

Let $E/\mathbb{Q}$ be an elliptic curve of conductor $N$. Work of Breuil, Conrad, Diamond, Taylor, and Wiles (see \cite{bcdt}, \cite{taylorwiles}, and \cite{wiles}) tells us that there exists a modular parametrization $\phi : X_0(N) \rightarrow E$ defined over $\mathbb{Q}$, where $X_0(N)/\mathbb{Q}$ is the modular curve associated to $\Gamma_0(N)$. Let $L/\mathbb{Q}$ be an imaginary quadratic field satisfying the Heegner hypothesis for $E/\mathbb{Q}$, i.e., all primes that divide the conductor $N$ of $E/\mathbb{Q}$ split in $L$. Let $\mathcal{O}_L$ be the ring of integers of $L$. The Heegner hypothesis implies that there exists an integral ideal $\mathcal{N}$ of $\mathcal{O}_L$ such that $\mathcal{O}_L/\mathcal{N} \cong \mathbb{Z}/N\mathbb{Z}$. Let $x_1 \in X_0(N)(\mathbb{C})$ be the point corresponding to the isogeny $\mathbb{C}/\mathcal{O}_L \longrightarrow \mathbb{C}/\mathcal{N}^{-1}$ (whose kernel is $\mathcal{N}^{-1}/\mathcal{O}_K \cong  \mathcal{O}_L/\mathcal{N} \cong \mathbb{Z}/N\mathbb{Z}$). The theory of complex multiplication tells us that $x_1 \in X_0(N)(H)$, where $H$ is the Hilbert class field of $L$. Finally, let $P_L=\text{Trace}_{H/L} ( \phi (x_1) )$, which is called the Heegner point associated with $L$ and is well defined up to sign and torsion.

Let $m$ be the Manin constant of $E/\mathbb{Q}$, let $c(E/\mathbb{Q})$ be the product of the Tamagawa numbers of $E/\mathbb{Q}$, and let $2u_L$ be the number of roots of unity in $L$. The following conjecture is due to Gross and Zagier.
 
 \begin{conjecture}\label{fullgrosszagierconjecture}(\cite[Conjecture (2.2) in Section V]{grosszagierpaper})
 If $P_L$ has infinite order in $E_L(L)$, then then $P_L$ generates a subgroup of finite index in $E_L(L)$ and this index equals $c(E/\mathbb{Q}) \cdot m \cdot u_L \cdot \sqrt{|\Sha(E_L/L)|}$.
 \end{conjecture}
 
 Since the index of $P_L$ in $E_L(L)$ is divisible by  $|E(\mathbb{Q})_{\text{tors}}|$, Conjecture \ref{fullgrosszagierconjecture} implies the following weaker conjecture.
 
 \begin{conjecture}\label{weakgrosszagierconjecture} (\cite[Conjecture (2.3) in Section V]{grosszagierpaper})
 If $P_L$ has infinite order in $E_L(L)$, then $|E(\mathbb{Q})_{\text{tors}}|$ divides $m\cdot c(E/\mathbb{Q}) \cdot u_L \cdot \sqrt{|\Sha(E_L/L)|}$.
 \end{conjecture}

Very recently Byeon, Yhee, and Kim (see \cite{bky} and \cite{bky2}) have proved the following theorem, which settles Conjecture \ref{weakgrosszagierconjecture} up to a power of $2$.

\begin{theorem}\label{byeonyheekimtheorem}
If $E/\mathbb{Q}$ is an elliptic curve with $E(\mathbb{Q})_{\text{tors}} \ncong \mathbb{Z}/2\mathbb{Z}$ or $\mathbb{Z}/4\mathbb{Z}$, then Conjecture \ref{weakgrosszagierconjecture} is true.
\end{theorem}

We end this section with the following corollary.

\begin{corollary}\label{corollaryheegner}
Let $E/\mathbb{Q}$ be an elliptic curve and $L \neq \mathbb{Q}(\sqrt{-3})$ be an imaginary quadratic field satisfying the Heegner hypothesis for $E/\mathbb{Q}$. Assume that the Heegner point $P_L$ has infinite order in $E_L(L)$. Then
\begin{enumerate}
    \item The quotient $|E_L(L)_{\mathrm{tors}}|/ |E(\mathbb{Q})_{\mathrm{tors}}|$ is equal to a power of $2$. In particular, $|E_L(L)_{\text{tors}}|$ divides $m\cdot c(E/\mathbb{Q}) \cdot u_L \cdot \sqrt{|\Sha(E_L/L)|}$, up to a power of $2$.
    \item If $E(\mathbb{Q})[2] \cong \{0\}$, then $|E_L(L)_{\text{tors}}| = |E(\mathbb{Q})_{\text{tors}}|$. In particular, $|E_L(L)_{\text{tors}}|$ divides $m\cdot c(E/\mathbb{Q}) \cdot u_L \cdot \sqrt{|\Sha(E_L/L)|}$.
\end{enumerate}

\end{corollary}
\begin{proof}
{\it Proof of Part $(i)$:} We first show that the quotient $|E_L(L)_{\mathrm{tors}}|/ |E(\mathbb{Q})_{\mathrm{tors}}|$ is a power of $2$. Write $L=\mathbb{Q}(\sqrt{d})$ for $d \in \mathbb{Z}$ square-free. Using Part $(i)$ of Theorem \ref{theoremtorsiongrowth} we find that $E_L(L)_{\mathrm{tors}}\setminus E(\mathbb{Q})_{\mathrm{tors}}$ cannot contain points of order $p >3$. If $d \neq \pm 3$, then there exists a prime $p \neq 3$ which ramifies in $L$. Therefore, using Part $(i)$ of Theorem \ref{theoremtorsiongrowth} we find that $E_L(L)_{\mathrm{tors}}\setminus E(\mathbb{Q})_{\mathrm{tors}}$ cannot contain points of order $3$. Since the case $d=3$ corresponds to a real quadratic field, we have proved that the quotient $|E_L(L)_{\mathrm{tors}}|/ |E(\mathbb{Q})_{\mathrm{tors}}|$ is a power of $2$. The last assertion of this part follows now immediately from Theorem \ref{byeonyheekimtheorem}.

{\it Proof of Part $(ii)$:} If $2$ divides $|E_L(L)_{\text{tors}}|$, then it follows from \cite[Part $(i)$ of Lemma 1.4]{laskalorenz} that $E(\mathbb{Q})[2] \ncong \{ \mathcal{O} \}$. Therefore, our assumption implies that $|E_L(L)_{\text{tors}}|$ is odd. However, by the previous part we know that the quotient $|E_L(L)_{\mathrm{tors}}|/ |E(\mathbb{Q})_{\mathrm{tors}}|$ is a power of $2$. Thus, we have that $|E_L(L)_{\text{tors}}| = |E(\mathbb{Q})_{\text{tors}}|$. Finally, the last assertion follows immediately from Theorem \ref{byeonyheekimtheorem}. This proves our corollary.
\end{proof}

\bibliographystyle{plain}
\bibliography{bibliography.bib}

\end{document}